\documentclass[11pt, reqno]{amsart}
\usepackage{amsmath, amsthm, a4, latexsym, amssymb}

\def\@rmrk#1#2{\refstepcounter
	{#1}\@ifnextchar[{\@yrmrk{#1}{#2}}{\@xrmrk{#1}{#2}}}

\makeatletter\@addtoreset{equation}{section}\makeatother

\newfont{\bfit}{cmbxti10 scaled 2000}
\newfont{\biggi}{cmr12 scaled 2000}

\newcommand{\eps}{\varepsilon}

\newcommand{\R}{\mathbb{R}}

\newcommand{\N}{\mathbb{N}}

\newcommand{\prob}{\mathbb{P}}

\newcommand{\me}{\mathbb{E}}

\renewcommand{\P}{\mathbb{P}}

\newcommand{\skrib}{{\mathcal B}}
\newcommand{\skric}{{\mathcal C}}
\newcommand{\skrid}{{\mathcal D}}

\newcommand{\skrie}{{\mathcal E}}

\newcommand{\skrig}{{\mathcal G}}
\newcommand{\skrih}{{\mathcal H}}

\newcommand{\skrik}{{\mathcal K}}

\newcommand{\skrim}{{\mathcal M}}
\newcommand{\skriw}{{\mathcal W}}

\newcommand{\skris}{{\mathcal S}}

\newcommand{\skriy}{{\mathcal Y}}

\newcommand{\sfrac}[2]{\mbox{$\frac{#1}{#2}$}}

\def\1{{\mathchoice {1\mskip-4mu\mathrm l}      
		{1\mskip-4mu\mathrm l}
		{1\mskip-4.5mu\mathrm l} {1\mskip-5mu\mathrm l}}}

\newcommand{\eq}{\begin{equation}}
	\newcommand{\en}{\end{equation}}

\newenvironment{Proof}
{\vskip0.1cm\noindent{\bf Proof. }{\hspace*{0.3cm}}}%
{\nopagebreak {\hspace*{\fill}\rule{2mm}{2mm}}\\ }

{\nopagebreak {\hspace*{\fill}\rule{2mm}{2mm}}\\ }

\renewcommand{\subsection}{\secdef \subsct\sbsect}
\newcommand{\subsct}[2][default]{\refstepcounter{subsection}
	\vspace{0.15cm}
	{\flushleft\bf \arabic{section}.\arabic{subsection}~\bf #1  }
	\nopagebreak\nopagebreak}
\newcommand{\sbsect}[1]{\vspace{0.1cm}\noindent
	{\bf #1}\vspace{0.1cm}}

\newtheorem{theorem}{Theorem}[section]
\newtheorem{lemma}[theorem]{Lemma}
\newtheorem{cor}[theorem]{Corollary}

\newtheoremstyle{thm}{1.5ex}{1.5ex}{\itshape\rmfamily}{}
{\bfseries\rmfamily}{}{2ex}{}

\newtheoremstyle{rem}{1.3ex}{1.3ex}{\rmfamily}{}
{\itshape\rmfamily}{}{1.5ex}{}
\theoremstyle{rem}
\newtheorem{remark}{{\slshape\sffamily remark}}[]

\refstepcounter{subsection}

\def\thebibliography#1{\section*{References}
	\list%
	{\arabic{enumi}.}
	{\settowidth\labelwidth{[#1]}\leftmargin\labelwidth
		\advance\leftmargin\labelsep
		\parsep0pt\itemsep0pt
		\usecounter{enumi}}
	\def\newblock{\hskip .11em plus .33em minus .07em}
	\sloppy                   
	\sfcode`\.=1000\relax}

\begin{document}
	\title[LLDP  and LDP  for Sub-critical Communication Networks]
	{\Large  Large  Deviations and  Information theory   for  Sub-Critical  for the Signal -to- Interference -Plus- Noise  Ratio  Randon Network Models}
	
	\author[]{}
	
	\maketitle
	\thispagestyle{empty}
	\vspace{-0.5cm}

	\centerline{\sc By  E.  Sakyi-Yeboah$^1$, P. S. Andam,  L.  Asiedu$^1$ and  K.  Doku-Amponsah$^{1,2} $}
	
	{$^1$Department  of  Statistics and Actuarial  Science, University of  Ghana,  BOX  LG 115, Legon,Accra}

	{$^2$ Email: kdoku-amponsah@ug.edu.gh}

	{$^2$ Telephone: +233205164254}
	\vspace{0.5cm}
	
	\renewcommand{\thefootnote}{}

	\footnote{\textit{Acknowledgement: }  This  Research work has been  supported  by funds  from  the  Carnegie Banga-Africa Project,  University  of  Ghana}
	\renewcommand{\thefootnote}{1}

	\vspace{-0.5cm}

	\begin{quote}{\small }{\bf Abstract.}
		The article obtains  large  deviation  asymptotic  for  sub-critical communication  networks  modelled  as   signal-interference-noise-ratio(SINR) random   networks. To achieve  this,  we  define  the  empirical  power  measure  and  the  empirical  connectivity  measure,  as well as  prove   joint large  deviation  principles(LDPs)  for the    two  empirical  measures on  two  different  scales. Using  the  joint LDPs,  we  prove  an  Asymptotic  equipartition property(AEP) for    wireless   telecommunication  Networks  modelled  as  the subcritical SINR random networks. Further,  we  prove a Local  Large  deviation  principle(LLDP) for the sub-critical  SINR random network. From the LLDPs,  we  prove  the large  deviation  principle, and  a  classical  McMillan  Theorem for  the stochastic SINR  model  processes.  Note that,  the  LDPs  for the  empirical measures  of this  stochastic  SINR random  network   model  were derived    on  spaces  of  measures  equipped  with  the  $\tau-$  topology, and  the  LLDPs  were deduced in  the  space  of  SINR  model  process  without  any  topological  limitations. We  motivate  the  study  by  describing   a  possible  anomaly  detection test  for SINR   random  networks.

	\end{quote}\vspace{0.5cm}

	
	\vspace{0.3cm}
	

	\vspace{0.3cm}
	
	
	\section{Introduction  }
In  telecommunication,  Wireless  networks  are usually  modelled  by  the  SINR random  networks.  In  the  SINR  random  network  model  two  nodes  are  deemed  to  communicate  if  SINR  is  bigger  than  a   certain  threshold  as  specified  by some  technical constant.  
 In the process of addressing the additional requirement imposed on wireless communication  networks, in particular, a higher availability of a highly accurate modeling of the SINR is required.  Example, each  transmission   may  be  equipped  with  some  battery  power  which  may  be  called  the  mark  of  the  node and  the quantity  SINR    defined  by  the  inclusion of  the  marks   in  the  definition.  Further study of  the  SINR  network model  has  shown  that  an  SINR model  of interference is a more realistic model of interference than the protocol model of interference: a receiver node receives a packet so long as the signal to interference plus noise ratio is above a certain threshold.  See,  Bakshi et  al.~\cite{BJN2017}.\\
	
 There  are  many  applications  of  large deviation  techniques  to the  SINR  networks, which are used  as   models  for  telecommunication  networks. Some  of these  applications  include, the  analysis  of  bi-stability  in  networks, such as  notorious  bi-stability  in  multiple  access  protocols  the  Aloha,  and the  stochastic  behaviour  of  ATM  the  admission control,  sizing  of  internal  buffers,  and  the  simulation  of  ATM models, see,\cite{We1995}. and  prevention of  cyber-attacks  on   wireless  telecommunication  networks,  see  example  \cite{PC2008}.\\

Cybersecurity  of  the  devices  in  a  telecommunication  system is  a  major   issue when the  devices  become  increasing  dependent  on  computer  and  other  local  networks.  And  an  anomaly detection in  the  devices  networks  is  key  to  avoiding  disruption  in  the  telecommunication  systems.  Cybersecurity  of  the  intelligent electronic  devices  in  telecommunication  substations  has  been  recognized  as  a  critical  issue  for  smooth  running  of  the  system. One  main approach  to  dealing  with  these  issues  is to  develop  new  technologies  to  detect  and  disrupt  any  malicious  activities  over  the  networks.An  Anomaly  detection  may  be  regarded  as  an  early  warning mechanism to  extract  relevant  cybersecurity events from  devices  locations  and  correlate  these  events.  Large  deviation principles   have  played  key  role  in  the  formulation  of  efficient  anomaly  inference  algorithm  for  systems such  as  power grid, Wireless Sensor  Network  systems  and   Telecommunication  systems. \\

	In  this  article,  we  prove   joint  large  deviation  principles  on  the  scales  $\lambda$  and  $\lambda^{2} a_{\lambda}$,  where  $\lambda$ is the  intensity  measure  of  the  underlining PPP  of  the  subcritical  SINR  model. See, \cite{SAD2020}  or  \cite{SKAD2020} or \cite{SAAD2020}   for  similar  results  fore  the dense  SINR  random  network  models. From  these  LDPs,  we  prove  an  asymptotic  equipartition  property; see example  \cite{SAD2020},  for  the  SINR  models.\\
	
	Further,  the study shows  a  LLDP for the  SINR  models. See example, \cite{SAD2020} and references therein. From  the  LLDP,  we deduce  asymptotic  bounds  on  the  cardinality  of  the  set  of  SINR  models  for  a  given  typical   empirical  marked  measure. In addition, the study shows that from  the LLDP  an  LDP for  the  SINR  modelled  processes.\\

	\subsection{Background}\label{SINR}\label{Sec1}

	This study set a dimension  $d\in\N$    and some  measureable  set  $\skrid\subset \R^k$    with reference to  the  Borel-$\sigma$ algebra  $\skrib(\R^d).$ 
	Given  $\lambda \mu:\skrid \to [0,1]$, an intensity measure and probability kernel density function from  $\skrid$   to  $\R^{+}$,  $\skrik$ and a  path loss  model, $\pi(\eta)=\eta^{-\alpha}, $   where  $\alpha \in \R^{+},$ and   
	some technical constraint; $\iota^{(\lambda)}, \zeta^{(\lambda)}:\R^{+}\to \R^{+}.$ The study defined  the  SINR network  model to as  follows:

	\begin{itemize}
		\item  We select  $\sigma=(\sigma_u)_{u\in I},$   a  Poisson  Point  Process  (PPP)  with rate measure  $\lambda \mu:\skrid \to [0,1]$.
		\item Given the  process  $\sigma,$  the locations,    each  $\sigma_u  $ is  assigned  a mark or  power  $\ell(\sigma_u)=\ell_u$  independently  according  to  the  kernel density function   $\skrik(\cdot \,,\,\sigma_u).$
		\item For any two set of marked points  $((\sigma_u,\ell_u),(\sigma_v,\ell_x))$  we link an edge  if and only if  $$SINR(\sigma_u,\sigma_v, \sigma) \ge \iota^{(\lambda)}(\ell_v) \mbox{  and      $SINR(\sigma_v,\sigma_u, \sigma)\ge \iota^{(\lambda)}(\ell_u),$}$$  where $$SINR(\sigma_v,\sigma_u, \sigma)=\frac{\ell_u \pi(\|\sigma_u-\sigma_v\|)}{N_0 +\zeta^{(\lambda)}(\ell_v)\sum_{u\in I\setminus\{v\} }\ell_u\pi(\|\sigma_u-\sigma_v\|)}$$
	\end{itemize}
	
We  let  $E$  denote  the  set   of  edges  in  the  SINR  random  network  and observe  $ Y^{\lambda}:=Y^{\lambda}(\ell, \sigma, \mu)=\Big\{[(\sigma_u,\ell_u), u\in I], \, E\Big\} $  under  the  joint  law  of  the  marked PPP  and    the  network.  In this article, we  call  $Y^{\lambda}$    an SINR  Network  model  and   $ (\sigma_u,\ell_u):= \sigma_u^{\lambda}$  as the   mark  of  site $u.$ Recall from \cite{SAD2020}  that  if  $N_0=0$, then 
	the connectivity function  of  the SINR random network model,   $ T^{\lambda}$, is  defined  as   $ T^{\lambda}((u,\ell_u),(v,\ell_v))= e^{-\lambda t_{\lambda}^{\skrid} ((u,\ell_u),(v,\ell_v)) },$  where
	
	$$t_{\lambda}^{\skrid} ((u,\ell_u),(v,\ell_v))= \int_{D} \Big[\sfrac{ \iota^{(\lambda)}(\ell_u) \zeta^{(\lambda)}(\ell_u)  }{\iota^{(\lambda)}(\ell_u) \zeta^{(\lambda)}(\ell_u)+(\|r\|^{\eta}/\|u-v \|^{\eta})} + \sfrac{ \iota^{(\lambda)}(\ell_v) \zeta^{(\lambda)}(\ell_v)  }{\iota^{(\lambda)}(\ell_v) \zeta^{(\lambda)}(\ell_v)+(\|r\|^{\eta}/\|v-u \|^{\eta})}\Big] \mu(dr).$$

	This article assumes that there exists  $a_{\lambda}$   and  a function  $t :\skrid\times \R_+\to (0,\infty)$ such that  $\lambda^{2}a_{\lambda}\to 0 $  and $$\displaystyle \lim_{\lambda\uparrow\infty}a_{\lambda}^{-1}T^{\lambda}((a,\ell_a),(b,\ell_b))=t((a,\ell_a),(b,\ell_b)).$$
	
	Sakyi-Yeboah et. al~\cite{SKAD2020} and Sakyi-Yeboah et. al~\cite{SAAD2020} investigates the  critical SINR network model (that is  $\lambda a_{\lambda}\to 1$) and super-critical SINR network model ( that is $\lambda a_{\lambda}\to \infty $) respectively . In  this articles, we  shall  focus this  study   on sub-critical SINR Networks( that is $\lim_{\lambda\to\infty}\lambda a_{\lambda}\to 0 $).\\
	
	For a given set  $\skrid)$  we  define  $\skris(\skrid)$  by

	\begin{equation}
		\skris(\skrid)=\cup_{x\subset \skrid}\Big\{x:\,\, |x\cap W|<\infty\,\, ,\mbox{for\, any  bounded  $W\subset \skrid$ }\Big \}.
	\end{equation}

	Let  $\skriw= \skris(\skrid\times\R_+)$   and    $\skrim(\skriw)$,  represent  the  space  of  positive  measures  on  the  space  $\skriw$   equipped  with  $\tau-$  topology. Note, $\skriw$ is a locally  finite  subset  of  the  set  $\skrid\times\R_+.$  See, example,  \cite{SKAD2020}.  Without  abuse  of  notation  we  shall  refer  to  $\skrim(\skriw\times\skriw)$  as  the  space  of  symmetric  measure  on  $\skriw\times\skriw$  endowed  with  the  $\tau-$ topology.
	For any SINR random network model  $Y^\lambda$  we  define a probability measure, the
	\emph{empirical power measure}, ~ $M_1^{Y^\lambda}\in\skrim(\skriw)$,~by
	$$M_1^{Y^\lambda}\big((a,\ell_a)\big ):=\frac{1}{\lambda}\sum_{u\in \skriw}\delta_{\sigma_u^{\lambda}}\big((a,\ell_a)\big)$$
	and a  finite measure, the \emph{empirical connectivity measure}
	$M_2^{Y^\lambda}\in\skrim(\skriw\times \skriw),$ by
	$$M_2^{Y^\lambda}\big((a,\ell_a),(b,\ell_b)\big):=\frac{1}{\lambda^2a_{\lambda}}\sum_{(u,v)\in E}[\delta_{(\sigma_u^{\lambda},\sigma_v^{\lambda})}+
	\delta_{(\sigma_v^{\lambda},\sigma_u^{\lambda})}]\big((a,\ell_a),(b,\ell_b)\big).$$
	
	It should be noted that the  total mass  $\|M_1^{Y^\lambda}\|$ of  the  empirical power  measure  is $\1$  and  total  mass  of  
	the empirical connect measure is
	$2|E|/\lambda^2a_{\lambda}$.

	\subsection{Motivation:
	Anomaly detection  in  spatial  networks}
	
	Consider,  SINR random  network model as  a  model  that  account  for  the  connectivity structure  of  the   Wireless  telecommunication networks (WTN). In  particular, consider the  subcritical  SINR random networks  as  model  for the  WTNs  since,  in  the  implementation,   the  multihop network  formed  by  the  sensor   nodes  may  adopt  a  network  structure.  The  network will  be  formed  randomly  according  to an  arbitrary   rule  that is  dependent  on  the  distances  between  the  device  locations. Assume  the  device  locations  are  marked  according  to  their  battery  power,  and  the  propagation of  events  is  un-directed on  the  network.  Our  objective  is  to  estimate  network  parameters  and  possible  identify  possible  deviations form  the  actual  values.
	
	For instance, given  a  long  sequence  of  realization $Y^{\lambda,k}$ of this  sub-critical  marked  SINR random  network, one would like to  approximate  parameter   of  the  model,  $\mu\times \skrik$  and  $t$,    by    taking  the  average  frequencies  of  the  corresponding  samples. In  particular,  if    $M_1^{Y^{\lambda,k}}$ and 	$M_2^{Y^{\lambda,k}}$;  the  empirical  power measure  and  the  empirical  connectivity  measure  of $Y^{\lambda}$, the  $k^{th}$  realization then 
	
	$$\displaystyle \lim_{k\to\infty}\sfrac{1}{k}\sum_{r=1}^{k}M_1^{Y^{\lambda,r}}(a,\ell_a)\to\mu\otimes\skrik(a,\,\ell_a)$$  and 
	
	$$\lim_{k\to\infty}\Big[\sfrac{1}{k}\sum_{r=1}^{k}M_2^{Y^{\lambda,r}}\Big((a,\ell_a),(b,\ell_b)\Big)/\sfrac{1}{k}\sum_{r=1}^{k}M_2^{Y^{\lambda,k}}(a,\ell_a)\otimes\sfrac{1}{k}\sum_{r=1}^{k}M_1^{Y^{\lambda,r}}(b,\ell_b)\Big]\to t\Big((a,\ell_a),(b,\ell_b)\Big),$$  with  probability  $1.$

	Assuming  that  we  have  estimated  $\mu\otimes \skrik$  and  $t.$  We   are  interested  in  a  test  that  determines  whether  a  particular  realization  $Y^{\lambda}$  is  typical  or  not. Thus,  we   want  to  differentiate  between  $\mu\times\skrik$  and  $t$ (Hypothesis $H_0$)  and  any  other  unknown  law (Hypothesis $H_1$). Theorem~\ref{main1a}  will be   the  bases of  providing generalized  Neyman-Pearson  criterion,See \cite[pp.96-100]{DZ1998},and  hence  an   anomaly  detection  test for  the sub-critical marked  SINR  random networks. \\

		This article is structured as follows:  Section~\ref{Sec2} presents  the  main  results; Theorem~\ref{main1a}, Theorem~\ref{main1b}, Theorem~\ref{main1c}, Corollary~\ref{cardinality}  and  Corollary~\ref{main2d}. In  Section~\ref{Sec3}   we prove  the main  results of  the  article, Theorem~\ref{main1a}.  Section~\ref{Sec4} provides the  proof  of  the  AEP,  see  Theorem~\ref{main1b}  and  Section~\ref{Sec5}; Proof  of  Theorem~\ref{main1c}, Corollary~\ref{cardinality} and  Corollary~\ref{main2d}. Lastly,   Section~\ref{Sec6} presents the conclusion to the article.

	\section { Main  Results}\label{mainresults}\label{Sec2}

	Theorem~\ref{main1a},  is  a  joint  large  deviation principle   for  the  empirical  measures  of  the SINR  network models.With reference  from  Subsection~\ref{Sec1}, we  recall  the  definition  of  $t_{\lambda}^{\skrid}$  as  
	
	$$t_{\lambda}^{\skrid} ((a,\ell_a),(b,\ell_b))= \int_{D} \Big[\sfrac{ \iota^{(\lambda)}(\ell_u) \zeta^{(\lambda)}(\ell_u)  }{\iota(\ell_u) \zeta(\ell_u)+(\|r\|^{\eta}/\|i-y \|^{\eta})} + \sfrac{ \tau^{(\lambda)}(\ell_v) \gamma^{(\lambda)}(\ell_v)  }{\tau^{(\lambda)}(\ell_v) \gamma^{(\lambda)}(\ell_v)+(\|r\|^{\eta}/\|y-x \|^{\eta})}\Big] \mu(dr)$$    and   note that	 $$t\beta\otimes\beta((a,\ell_a),(b,\ell_b))):=t((a,\ell_a),(b,\ell_b))\mu((a,\ell_a))\mu((b,\ell_b)).$$  

	\begin{theorem}	\label{main1a}
		Let   $Y^{\lambda}$  is  a sub-critical marked SINR network model  with rate measure
		$\lambda \mu:\skrid \to [0,1]$ and   a  power transition kernel function  $\skrik(\cdot, y)=ce^{-cy},  y>0$  and  path  loss  function   $\pi(\eta)=\eta^{-\alpha}, $  for  $\alpha>0.$   Thus, the link  kernel function  $T^{\lambda} $  of   $Y^{\lambda}$  satisfies  $a_{\lambda}^{-1}T^{\lambda}\to t$  and  $\lambda a_{\lambda} \to 0.$  Then, as  $\lambda\to \infty$,  the  pair of  measures  $(M_1^{Y^\lambda},M_2^{Y^\lambda})$  satisfies a  large  deviation principle  in the  space 
		$\skrim(\skriw)\times \skrim(\skriw\times\skriw)$  
		
		\begin{itemize}
			\item [(i)]	 with speed $\lambda$   and  a  good  rate  function

			\begin{equation}
				\begin{aligned}
					I^{1}\big(\beta,\phi\big)= \left\{\begin{array}{ll}H\Big(\beta\Big |\mu\otimes \skrik\Big)&\,\,\mbox{ if $ \phi=t\beta\otimes\beta $ }\\
						\infty & \mbox{elsewhere.}
					\end{array}\right.
				\end{aligned}
			\end{equation}
			
			\item[(ii)]with  speed   $\lambda^2a_{\lambda}$  and  good rate function
			
		\begin{equation}
		\begin{aligned}
		I^{2}\big(\beta,\phi\big)= \left\{\begin{array}{ll}	\skrih(\phi\|t\beta\otimes\beta),&\,\,\mbox{ if $ \beta=\mu\otimes \skrik $ }\\
		\infty & \mbox{elsewhere.}
		\end{array}\right.
		\end{aligned}
		\end{equation}
		\end{itemize}
		
		where

		\begin{equation}
			\begin{aligned}
				\skrih(\phi\|t\beta\otimes\beta):= \left\{\begin{array}{ll} H(\phi\,\|\,t\beta\otimes \beta)+\Big(\|t\beta \otimes\beta\|-\|\phi\|\Big),  & \mbox{if  $\|\phi\|>0.$  }\\
					\infty & \mbox{elsewhere.}		
				\end{array}\right.
			\end{aligned}
		\end{equation}

	\end{theorem}

	\begin{theorem}	\label{main1b}
		Suppose $Y^{\lambda}$ be a sub-critical marked SINR network model  with rate measure
		$\lambda \beta:D \to [0,1]$ and   a   power  probability  function  $\skrik(\cdot, y)=ce^{-cy},  y>0$   and  path  loss  function   $\pi(\eta)=\eta^{-\alpha}, $  for  $\alpha> 0.$  Thus, the connectivity  probability  $T^{\lambda} $  of   $Y^{\lambda}$  satisfies  $a_{\lambda}^{-1}T^{\lambda}\to t $  and  $\lambda a_{\lambda} \to 0.$   Suppose  the  sequence $a_{\lambda}$  of  $Y^{\lambda}$  is  such  that  $\lambda a_{\lambda}\,\log\lambda\to 0$  and  $a_{\lambda}/\log\lambda\to-1.$	Then,  we  have  
		$$\lim_{\lambda\to\infty}\prob\Big\{\Big|-\frac{1}{a_{\lambda}\lambda^2\log\lambda}\log P(Y^{\lambda})-\me_f\Big[t((\cdot,\cdot),(\cdot,\cdot))\Big]\Big|\ge \eps\Big\}=0,$$
		
		where  the  expectation  was  taken  with  respect  to the distribution function  $$f((x,\ell_x),(y,\ell_y))=c^2e^{-c(x+y)}\mu(d\ell_x)\mu(d\ell_y)dxdy,\,\mbox{ $x>0,y>0,\ell_x>0,\ell_y>0 .$} $$
		
	\end{theorem}

Note  that  the  $H(f):=\me_f\Big[t((\cdot,\cdot),(\cdot,\cdot))\Big]$  is  an  entropy.\\

{\bf Interpretation:} To  transmit  information  contain  in  a  large  SINR random  network  modls  one  require  with  a  large  probability  
$$-\lambda^2 a_{\lambda}\log \lambda\Big[H(f)\Big]/\log2 \mbox{ \bf bits.}$$

	Let  $\skrig$   be  the  set  of all  SINR  networks  with rate measure
	$\lambda \mu:\skrid \to [0,1]$  and  state  the  Local  Large  deviation principle  as  follows:
	
	\begin{theorem}	\label{main1c}
		Suppose   $Y^{\lambda}$  is  a sub-critical marked SINR network model  with rate measure
		$\lambda \mu:\skrid \to [0,1]$ and   a   mark  transition kernel   $\skrim(y)=c{e}^{-cy}, y>0$   and  path  loss  function   $\pi(\eta)=\eta^{-\alpha}, $  for $\eta>0$ and   $\alpha>0.$  Thus, the link  probability  $T^{\lambda} $  of   $Y^{\lambda}$  satisfies  $a_{\lambda}^{-1}T^{\lambda}\to t $  and  $\lambda a_{\lambda} \to 0.$
		Then, 
		\begin{itemize}
			\item  for  any  functional  $\phi\in\skrig$  and    a  number $\eps>0$,  there  exists  a  weak  neighbourhood $B_{\phi}$  such  that    
			$$\P_{\beta}\Big\{Y^{\lambda}\in \skrig\,\Big|\, M_2^{Y^\lambda}\in B_{\phi}\Big\}\le e ^{-\sfrac{1}{2}\lambda^2 a_{\lambda} 	\skrih(\phi\|t\beta\otimes\beta)-\lambda a_{\lambda}\eps},\,	\mbox{where  $\beta=\mu\otimes\skrik.$ }$$  
			\item  for  any    $\phi\in \skrig_{\beta}$,  a  number  $\eps>o$  and  a  fine  neighbourhood $B_{\phi} $,  we  have  the  compute: 
			$$\P_{\beta}\Big\{Y^{\lambda}\in \skrig\,\Big|\, M_2^{Y^\lambda}\in B_{\phi}\Big\}\ge e^{-\sfrac{1}{2}\lambda ^2a_{\lambda}	\skrih(\phi\|t\beta\otimes\beta)+\lambda_{\lambda} a_{\lambda}\eps},\,	\mbox{where  $\beta=\mu\otimes\skrik.$ }$$
		\end{itemize}

	\end{theorem}
	
	For the given telecommunication  network model, we  define  an  entropy as  $h: \skrim(\skriw\times\skriw)\to [0.\infty]$   by   
	
		\begin{equation}\label{equ3}
		h(\phi):=\Big(\|\phi\|-\|\lambda\beta\otimes\beta\|-\Big\langle \phi\, ,\,\log \sfrac{\phi}{\|t\beta\otimes\beta\|}\Big\rangle\Big)/2,\,\,\mbox{	where  $\beta=\mu\otimes\skrik.$}
	\end{equation}
	
	\begin{cor}[McMillian Theorem]\label{cardinality}\label{main2c} Let $Y^\lambda$ be a sub-critical marked SINR network model  with rate measure
		$\lambda \mu:\skrid \to [0,1]$ and   a   mark transition kernel  $\skrik(\cdot,\, y)=c{e}^{-cy}, y>0$   and  path  loss  function   $\pi(\eta)=\eta^{-\alpha}, $  for  $\eta>0$  and	where  $\beta=\mu\otimes\skrik.$ $\alpha>0.$  Thus, the link  probability  $T^{\lambda} $  of every  $Y^{\lambda}\in \skrig$  satisfies  $a_{\lambda}^{-1}T^{\lambda}\to t $  and  $\lambda a_{\lambda} \to 0.$

		\begin{itemize}
			
			\item[(u)] For  any empirical link  measure  $\phi$   on  $\skriw\times\skriw$   and  $\eps>0,$  there  exists  a neighborhood  $B_{\phi}$  such  that
			$$ Card\Big(\big\{Y^{\lambda}\in\skrig\,| \,M_2^{Y^\lambda}\in D_{\phi}\big\}\Big)\ge e^{\lambda^2 a_{\lambda}(h(\phi)-\eps\big)}.$$
			\item[(ii)] for any  neighborhood  $B_{\rho}$   and  $\eps>0,$  we  have
			$$Card\Big(\big\{Y^{\lambda}\in\skrig\,|\, M_2^{Y^\lambda}\in B_{\phi}\big\}\Big)\le e^{\lambda^2 a_{\lambda}(h(\phi)+\eps\big)},$$
		\end{itemize}
	\end{cor}
	
	where $Card(\gamma)$  means  the  cardinality  of  $\gamma.$
	
	\begin{remark}
		Given $\phi=t\beta\otimes\beta,$  we  have	$\displaystyle Card\Big(\Big\{j\in\skrig\,\Big\}\Big)\approx e^{\lambda^2\,a_{\lambda}\|t\beta\otimes\beta\|\skrih\big(t\beta\otimes\beta/\|t\beta\otimes\beta\|\big)},$   where  $\beta=\mu\otimes\skrik.$
		
	\end{remark}
	
	{\bf  Interpretation:}
	Note from Corollary~\ref{cardinality}  that, for  the typical  empirical  connectivity measure, $t\mu^2\otimes\skrik^2,$  the cardinality  of the  space of SINR  models is  nearly equal to $\displaystyle e^{\lambda^2 a_{\lambda}\|t\mu^2\otimes\skrik^2\|H\big(t\mu^2\otimes\skrik^2/\|t\mu^2\otimes\skrik^2\|\big)}.$
 The  next theorem  is  the  LDP  for  the  SINR  random  network  processes.

	\begin{cor}\label{randomg.LDM}\label{main2d}
		Let   $Y^{\lambda}$  be  a sub-critical marked SINR random  network model  with rate measure
		$\lambda \mu:\skrid \to [0,1]$ and   a  mark  kernel  function  $\skrik(\cdot,\,y)=c{e}^{-cy}, y>0$   and  path  loss  function   $\pi(\eta)=\eta^{-\alpha}, $  for  $\alpha>0.$  Thus, the link  probability  $T^{\lambda} $  of   $Y^{\lambda}$  satisfies  $a_{\lambda}^{-1}T^{\lambda}\to t$  and  $\lambda a_{\lambda} \to 0.$ 
		\begin{itemize}
			\item  Let $U$  be  closed  subset  $\skrig$.  Then  we  have  
			$$\limsup_{\lambda\to\infty}\frac{1}{\lambda^2 a_{\lambda}}\log \P_{\mu\times\skrik}\Big\{Y^{\lambda}\in \skrig\,\Big|\, M_2 ^{Y^\lambda}\in U\Big\}\le -\sfrac{1}{2}\inf_{\phi\in U}\Big\{\skrih(\phi\|t\mu\times\skrik\otimes \mu\times\skrik)\Big\}$$
			\item  Let  $O$  be  open  subset  $\skrig$.  Then  we  have  
			$$\liminf_{\lambda\to\infty}\frac{1}{\lambda^2 a_{\lambda}}\log \P_{\beta}\Big\{Y^{\lambda}\in \skrig\,\Big|\, M_2^{Y^\lambda}\in O\Big\}\ge -\sfrac{1}{2}\inf_{\phi\in O}\Big\{\skrih(\phi\|t\mu\times\skrik\otimes \mu\times\skrik)\Big\}.$$
		\end{itemize}
		
	\end{cor}

	\section{Proof  of Main  Results}\label{Sec3}
	\subsection{Proof  of   Theorem~\ref{main1a}(i)}\label{proofmain}
	
	Suppose   $W_1,...,W_n$   is  a  decomposition  of  the  space  $\skrid\times \R_{+}.$   Note  that,  for  every $(u,v)\in A_x\times A_y,\, x,y=1,2,3,...,n,$  $\lambda M_2^{Y^\lambda}(u,v)$  given  $\lambda M_1^{Y^\lambda}(u)=\lambda\beta(u)$  denotes a number of bernoulli trial with  parameters  $\lambda^2\beta(u)\beta(v)/2$  and  $T^{\lambda}(u,v).$ Consider  $\skrik$ to represent as the  gamma  distribution  with  mean $1/c.$  With reference to the  function  $t_{\lambda}^{\skrid}  $  from  the  preceding  sections, we  observe  that
	Lemma~\ref{main1c}  is  fundamental    in  the  application  of  the  Gartner-Ellis  Theorem.  See \cite{DZ1998}.  
	\begin{lemma}\label{randomg.LDM1b}\label{main1ca}
		
		Suppose   $Y^{\lambda}$  is  a sub-critical marked SINR random model  with rate measure
		$\lambda \mu:\skrid \to [0,1]$ and   a   power  probability  function  $\skrik(\cdot,\,y)=c{e}^{-cy}, y>0$   and  path  loss  function   $\pi(\eta)=\eta^{-\alpha}, $  for  $\eta>0$ and  $\alpha>0.$  Thus, the link  probability  $T^{\lambda} $  of   $Y^{\lambda}$  satisfies  $a_{\lambda}^{-1}T^{\lambda}to t$  and  $\lambda a_{\lambda} \to 0.$
		Suppose $Y^{\lambda}$  be  a sub-critical SINR network model,  conditional  on the  event  $M_1^{Y^\lambda}=\beta.$  Let $ q:\skriw\times \skriw\to \R$ be  bounded  function.  Then,
		
		$$\begin{aligned}\lim_{\lambda\to\infty}\frac{1}{\lambda}\log\me \Big\{e^{\lambda\langle q, \, M_2^{Y^\lambda}\rangle }\Big | M_1^{Y^\lambda}=\beta\Big\}&=\frac{1}{2}\lim_{n\to\infty}\sum_{y=1}^{n}\sum_{x=1}^{n}\Big\langle q,\, t\beta\otimes\beta\Big\rangle _{A_x \times A_y}\\
			&=\frac{1}{2}\Big\langle q,\, t\beta\otimes\beta\Big\rangle _{\skriw \times \skriw}.
		\end{aligned}$$
	\end{lemma}	
	\begin{Proof}
		Now  we  observe  that 
		$$ \me\Big \{ e^{\int \int \lambda q(u,v)M_2^{Y^\lambda}(du,dv)/2} \Big |M_1^{Y^\lambda}=\beta\Big \}= \me\Big\{\prod_{u\in \skriw} \prod_{v \in \skriw} e^{\lambda q(u,v)_2^{\lambda}(du,dv)/2} \Big\}$$
		
		$$ \me\Big \{\prod_{u \in \skriw} \prod_{v \in \skriw} e^{q(u,v)\lambda M_2^{Y^\lambda}(du,dv/2)}\Big\} = \prod_{x=1}^{n} \prod_{y=1}^{n} \prod_{u \in W_x} \prod_{v \in W_y} \me\Big\{e^{q(u,v) \lambda M_2^{Y^\lambda}(du,dv)/2  }\Big \} $$
		
		$$ \log \Big \{e^{\lambda \langle q,M_2^{Y^\lambda}\rangle/2}\Big|M_1^{Y^\lambda}=\beta\Big\} = \sum_{y=1}^{n} \sum_{x=1}^{n}\int_{W_y}\int_{W_x}\log\Big[1-T^{\lambda}(u,v)+T^{\lambda}(u,v)e^{q(u,v)/\lambda a_{\lambda}}\Big]^{\lambda^{2} \beta\otimes\beta(du,dv)/2}+o(n) $$

		Introducing the dominated convergence theorem 
		$$   \frac{1}{\lambda} \log E \{e^{\lambda \langle q,M_2^{Y^\lambda}\rangle/2 } \mid M_1^{Y^\lambda}=\beta\} = \frac{1}{\lambda}\sum_{y=1}^{n} \sum_{x=1}^{n} \int_{W_x} \int_{W_y} \log\Big [1-\big(1-e^{q(u,v)/\lambda a_{\lambda}}) T^{\lambda}(u,v)  \Big]^{ \lambda^2 \beta\otimes\beta(du,dv)/2} +o(n)/\lambda$$
		$$  \frac{1}{\lambda} \log \me \{e^{\lambda \langle q,M_2^{Y^\lambda}\rangle /2} \big| M_1^{Y^\lambda}=\beta\} = \lim_{\lambda \rightarrow \infty } \sum_{y=1}^{n} \sum_{x=1}^n\int_{W_x} \int_{W_y} \log \Big[1+q(u,v) t(u,v)/\lambda +o(\lambda)/\lambda   \Big]^{ \lambda \beta\otimes\beta(du,dv)/2}+o(n)/\lambda $$

		$$\lim_{\lambda \rightarrow \infty}\frac{1}{\lambda} \log \me\{e^{\lambda \langle q,M_2^{Y^\lambda}\rangle/2} \mid M_1^{Y^\lambda}=\beta\} 
		=\frac{1}{2}\sum_{y=1}^{n}\sum_{x=1}^{n} \Big\langle g,\, t\beta\otimes\beta\Big\rangle _{W_x \times W_y} $$
		
		$$ \begin{aligned}
			\lim_{\lambda \rightarrow \infty} \frac{1}{\lambda} \log \me \{e^{\lambda \langle q,M_2^{Y^\lambda}\rangle/2 } \Big |M_1^{Y^\lambda}=\beta\}& =\frac{1}{2} \lim_{n \rightarrow \infty } \sum_{y=1}^{n}\sum_{x=1}^{n} \Big\langle q,\, t\beta\otimes\beta\Big\rangle _{W_x \times W_y}\\
			&	=\frac{1}{2} \Big\langle q,\, t\beta\otimes\beta\Big\rangle _{\skriw \times \skriw} . 
		\end{aligned}$$

		Hence,by the Gartner-Ellis theorem, conditional  on the  event $\Big\{M_{1}^{Y^\lambda}= \beta\Big\}$, $M_2^{Y^\lambda}$ obey a  large  deviation  principle with speed $\lambda$  and variational formulation of  the  rate function
		$$ I_{\beta}(\phi) = \frac{1}{2}\sup_{q} \Big\{  \Big\langle q,\, \phi\Big\rangle _{\skriw \times \skriw}-  \Big\langle q,\, t\beta\otimes\beta\Big\rangle _{\skriw\times \skriw}\Big\}$$ 
		
		the solution can be found,  see  example   \cite{DA2012},	would obviously   reduces  to  the good rate function  as such 
		\begin{equation}
			I_{\beta}(\phi)= 0.
		\end{equation}
		
	\end{Proof}

	\subsection{Proof  of   Theorem~\ref{main1a}(ii) }\label{proofmain}
	
	Analogously  we consider  $W_1,...,W_n$  as  decomposition  of  the  space  $\skrid\times \R_{+}.$   We  refer to  $t_{\lambda}^{D}  $   and  observe  that,  Lemma~\ref{main1da}  will  play an important role    in  the  application  of  the  Gartner-Ellis  Theorem.  See,  \cite{DZ1998}.

	\begin{lemma}\label{randomg.LDM1b}\label{main1da}	Let   $Y^{\lambda}$  be  a sub critical powered SINR  network  with rate measure
	$\lambda \mu:\skrid \to [0,1]$ and   a   power  probability  function  $\skrik(y)=c{e}^{-cy}, y>0$   and  path  loss  function   $\pi(\eta)=r^{-\alpha}, $  for  $\eta>0$  and  $\alpha>0.$  Thus, the link  probability  $T^{\lambda} $  of   $Y^{\lambda}$  satisfies  $a_{\lambda}^{-1}T^{\lambda}\to t$  and  $\lambda a_{\lambda} \to \infty.$
	Let   $Y^{\lambda}$  be  a sub-critical SINR network,  conditional  on the  event  $M_1^{Y^\lambda}=\beta.$  Let $ q:\skriw\times \skriw\to \R$ be  bounded  function.  Then,
		
		$$\begin{aligned}\lim_{\lambda\to\infty}\frac{1}{\lambda^2 a_{\lambda}}\log\me \Big\{e^{\lambda^2 a_{\lambda}\langle q, \, M_2^{Y^\lambda}\rangle }\Big | M_1^{Y^\lambda}=\beta\Big\}&=-\frac{1}{2}\lim_{n\to\infty}\sum_{y=1}^{n}\sum_{x=1}^{n}\Big\langle 1-e^{q},\, t\beta\otimes\beta\Big\rangle _{W_x \times W_y}\\
			&=-\frac{1}{2}\Big\langle 1-e^{q},\, t\beta\otimes\beta\Big\rangle _{\skriw \times \skriw}.
		\end{aligned}$$
	\end{lemma}	
	\begin{Proof}
		Now  we note that
		$$ \me\Big \{ e^{\int \int \lambda^2 a_{\lambda} q(u,v)M_2^{Y^\lambda}(du,dv)/2} \Big |M_1^{Y^\lambda}=\beta\Big \}= \me\Big\{\prod_{i \in \skriw} \prod_{j \in \skriw} e^{\lambda^2 a_{\lambda} q(u,v)M_2^{Y^\lambda}(du,dv)/2} \Big\}$$
		
		$$ \me\Big \{\prod_{i \in \skriw} \prod_{j \in \skriw} e^{q(u,v)\lambda M_2^{Y^\lambda}(du,dv/2)} = \prod_{x=1} \prod_{y=1} \prod_{i \in W_x} \prod_{j \in W_y} \me\Big\{e^{\lambda^2 a_{\lambda}q(u,v)  M_2^{Y^\lambda}(du,dv)/2  }\Big \}\times e^{o(n)} $$
		
		$$ \log \Big \{e^{\lambda^2 a_{\lambda}\langle q,M_2^{Y^\lambda}\rangle/2}\Big|M_1^{Y^\lambda}=\beta\Big\} = \sum_{y=1}^{n} \sum_{x=1}^{n}\int_{W_y}\int_{W_x}\log\Big[1-T^{\lambda}(u,v) )+T^{\lambda}(u,v) e^{q(u,v)}\Big]^{\lambda^{2} \beta\otimes\beta (du,dv)/2}+o(n) $$

		Using the  dominated convergence theorem 
		$$   \frac{1}{\lambda^2 a_{\lambda}} \log E \{e^{\lambda \langle q,M_2^{Y^\lambda}\rangle/2 } \mid M_1^{Y^\lambda}=\beta\} = \frac{1}{\lambda^2 a_{\lambda}}\sum_{y=1} \sum_{x=1} \int_{W_x} \int_{W_y} \log\Big [1-\big(1-e^{q(u,v)}) T^{\lambda}(u,v)  \Big]^{ \lambda^2 \beta\otimes\beta(du,dv)/2} +o(n)/\lambda^2a_{\lambda}$$ 
		$$  \frac{1}{\lambda^2 a_{\lambda}} \log \me \{e^{\lambda \langle q,M_2^{Y^\lambda}\rangle /2} \big | M_1^{Y^\lambda}=\beta\} = \lim_{\lambda \rightarrow \infty } \sum_{y=1} \sum_{x=1}\int_{W_x} \int_{W_y} \log \Big[1-(1-e^{q(u,v)}) T^{\lambda}(u,v)  \Big]^{ \lambda \beta\otimes\beta (du,dv)/2} +o(n)/\lambda^2a_{\lambda}$$
		
		$$\lim_{\lambda \rightarrow \infty} \frac{1}{\lambda^2 a_{\lambda}} \log \me \Big\{e^{\lambda \langle q,M_2^{Y^\lambda}\rangle /2} \big| M_1^{Y^\lambda}=\beta\Big \} =- \frac{1}{2} \sum_{y=1} \sum_{x=1} \int_{W_x} \int_{W_y}\Big [(1-e^{q(u,v)})t(u,v)\beta\otimes\beta(du,dv)\Big] $$
		
		$$\lim_{\lambda \rightarrow \infty}\frac{1}{\lambda^2 a_{\lambda}} \log \me\{e^{\lambda \langle q,M_2^{Y^\lambda}\rangle/2} \big | M_1^{Y^\lambda}=\beta\} 
		=- \frac{1}{2}\sum_{y=1}^{n}\sum_{x=1}^{n} \Big\langle 1-e^{q},\, t\beta\otimes\beta\Big\rangle _{W_x \times W_y} $$
		
		$$ \begin{aligned}
			\lim_{\lambda \rightarrow \infty} \frac{1}{\lambda^2 a_{\lambda}} \log \me \{e^{\lambda \langle q,M_2^{Y^\lambda}\rangle/2 } \Big |M_1^{Y^\lambda}=\beta\}& =-\frac{1}{2} \lim_{n \rightarrow \infty } \sum_{y=1}^{n}\sum_{x=1}^{n} \Big\langle 1-e^{q},\, t\beta\otimes\beta\Big\rangle _{W_x \times W_y}\\
			&	=-\frac{1}{2} \Big\langle 1-e^{q},\, t\beta\otimes\beta\Big\rangle _{\skriw \times \skriw}  
		\end{aligned}$$

		Hence,by the Gartner-Ellis theorem, conditional  on the  event $\Big\{M_{1}^{Y^\lambda}= \beta\Big\}$, $M_2^{Y^\lambda}$ obey a  large  deviation  principle with speed $\lambda$  and variational  formulation of  the  rate function is  given  by 
		$$ I_{\beta}(\phi) = \frac{1}{2}\sup_{q} \Big\{  \Big\langle q,\, \phi\Big\rangle _{\skriw \times \skriw}+  \Big\langle 1-e^{q},\, t\beta\otimes\beta\Big\rangle _{\skriw\times \skriw}\Big\}$$ 
		
		which   when solved,  see  example   \cite{DA2012},	will clearly   reduce  to  the good rate function  given by 
		\begin{equation}
			I_{\beta}(\phi)=  \frac{1}{2}\skrih(\phi\|t \beta\otimes\beta).
		\end{equation}
		
	\end{Proof}

	\subsection{ Proof of  Theorem~\ref{main1a}(ii)    by  Method  of  Mixtures.}\label{Sec4}For any $\lambda\in (0,\infty)$ we define
	$$\begin{aligned}
		\skrim_{\lambda}(\skriw) & := \Big\{ \beta\in \skrim(\skriw) \, : \, \lambda\beta(u) \in \N \mbox{ for all } u\in \skriw\Big\},\\
		\tilde \skrim_{\lambda }(\skriw\times \skriw) & := \Big\{ \phi\in
		\skrim(\skriw\times \skriw) \, : \, 
		\lambda \,\phi(u,v) \in \N,\,  \mbox{ for all } \, u,v\in \skriw
		\Big\}\, .
	\end{aligned}$$
	
	We denote by
	$\Upsilon_{\lambda}:=\skrim_{\lambda }(\skriw)$
	and
	$\Upsilon:=\skrim(\skriw)$.
	We  write 
	$$\begin{aligned}
		P_{ \beta_{\lambda}}^{(\lambda)}(\phi_{\lambda}) & := \prob\big\{M_2^{Y^\lambda}=\phi_{\lambda} \, \big| \, M_1^{Y^\lambda}=\beta_{\lambda}\big\}\, ,\\
		P^{(\lambda)}(\beta_{\lambda}) & :=
		\prob\big\{M_1^{Y^\lambda}=\beta_{\lambda}\big\}
	\end{aligned}$$

	Th joint distribution of $M_1^{Y^\lambda}$ and $M_2^{Y^\lambda}$ is
	the mixture of $P_{ \beta_{\lambda}}^{(\lambda)}$ with
	$P^{(\lambda)}(\beta_{\lambda}),$    as follows: 
	\begin{equation}\label{randomg.mixture}
		d\tilde{P}^{\lambda}( \beta_{\lambda}, \ell_{\lambda}):= dP_{	\beta_n}^{(\lambda)}(\ell_{\lambda})\, dP^{(\lambda)}( \beta_{\lambda}).\,
	\end{equation}

	(Biggins, Theorem 5(b), 2004) provides condition for the validity of
	large deviation principles for the mixtures and for the goodness of
	the rate function if individual large deviation principles are
	known. The following three lemmas ensure validity of these
	conditions.
	
	Note that the  family of
	measures $({P}^{(\lambda)} \colon \lambda\in(0,\infty))$  is  exponentially tight on
	$\Upsilon.$

	\begin{lemma}[] \label{Com4}

		\begin{itemize}
			
			\item[(u)] 	The  family of
			measures $(\tilde{P}^{\lambda} \colon \lambda\in(0,\infty))$  is  exponentially tight on
			$\Upsilon\times\tilde\skrim(\skriw\times \skriw).$
			
			\item[(ii)] The  family	measures $(T^{\lambda} \colon \lambda\in(0,\infty))$  is  exponentially tight on
			$\Upsilon\times\skrim(\skriw\times \skriw).$
		\end{itemize}
	\end{lemma}
	
	We  refer to  \cite[Lemma~4.3]{SAD2020} for  similar  proof  for  Large  Deviation  Principle on  the  scale  $\lambda^2$
	
	Define the function
	$I_{sc}^2,I_{sc}^1\colon{\Upsilon}\times\skrim(\skriw\times \skriw)\rightarrow[0,\infty],$  by

	\begin{equation}
		\begin{aligned}
			I^{1}\big(\beta,\phi\big)= \left\{\begin{array}{ll}H\Big(\beta\Big |\mu\otimes\skrik \Big)&\,\,\mbox{ if $ \phi=t\beta\otimes\beta $ }\\
				\infty & \mbox{otherwise.}
			\end{array}\right.
		\end{aligned}
	\end{equation}

	\begin{equation}
		I^{2}\big(\beta,\nu\big)= \frac{1}{2} \skrih\Big(\nu\|t\beta\otimes\beta\Big).
	\end{equation}
	
	\begin{lemma}[]\label{Com5}
		\begin{itemize}
			
			\item[(u)]	$I^1$ is lower semi-continuous.
			\item  [(ii)] $I^2$  is  lower  semi-continuous.
		\end{itemize}
	\end{lemma}

	By (Biggins, Theorem~5(b), 2004) the two previous lemmas,  the  LDP  for  the  empirical  power  measure, see,  \cite[Theorem~2.1]{SAD2020} and the
	large deviation principles we have established
	Theorem~\ref{main1a} ensure
	that under $(\tilde{P}^{\lambda})$    and   $T^{\lambda}$ the random variables $(\beta_{\lambda}, \ell_{\lambda})$   satisfy a large deviation principle on
	$\skrim(\skriw) \times \skrim(\skriw\times \skriw)$ and   	$\Upsilon\times\skrim_{\lambda}(\skriw\times \skriw)$  on  the  speeds  $\lambda$ and  $\lambda^2 a_{\lambda}$ with good rate functions  $I^1$   and $I^2$  respectively,  which  ends  the  proof of  Theorem~\ref{main1a}.

	\section{Proof of   Theorem~\ref{main1b} by Large  deviations }\label{Sec4}

	To  prove the  Shannon-Mcmillian Breiman (SMB)  or  the  AEP,  we first prove  a  weak  law  of large  numbers (WLLN) for   the empirical  marked  measure  and  the  empirical connectivity  measure  of  the  SINR network model.
	\begin{lemma}\label{WLLN}	Let   $Y^{\lambda}$  be a sub-critical marked SINR  model  with rate measure
		$\lambda \mu:\skrid \to [0,1]$ and   a   marked transition  function  $\skrik(\cdot, y)=c{e}^{-cy}, y>0$   and  path  loss  function   $\pi(\eta)=\eta^{-\alpha}, $  for  $\alpha>0.$  Thus, the link  probability  $T^{\lambda} $  of   $Y^{\lambda}$  satisfies  $a_{\lambda}^{-1}T^{\lambda}\to t$  and  $\lambda a_{\lambda} \to 0.$
		Then,	$$\lim_{\lambda\to\infty}\P\Big\{\sup_{(a,\ell_a)\in\skriw}\Big|M_1^{Y^\lambda}(a,\ell_a)-\mu\otimes \skrik(a,\ell_a) \Big|>\eps\Big\}=0$$ and  
		$$\lim_{\lambda\to\infty}\P\Big\{\sup_{([y_u,\ell_u],[y_v,\ell_v])\in\skriw\times\skriw}\Big|M_2^{Y^\lambda}([u,\ell_u],[y_v,\ell_v])-t\mu\otimes \skrik\times \mu\otimes \skrik([y_u,\ell_u],[y_v,\ell_v]) \Big|>\eps\Big\}=0$$
	\end{lemma}
	
	\begin{proof}
		Let
		$$ U_{1,\skriw}=\Big\{\beta:\sup_{(a,\ell_a)\in\skriw}|\beta(a,\ell_a)-\mu\otimes \skrik(a,\ell_a)|>\eps\Big\},$$ $$U_{2,\skriw}=\Big\{\phi:\sup_{([u,\ell_u],[j,\ell_v])\in\skriw\times\skriw}|\phi([y_u,\ell_u],[y_v,\ell_v])-t\mu\otimes \skrik\times \mu\otimes \skrik([y_u,\ell_u],[y_v,\ell_v])|>\eps\Big \}$$
		
		and  $U_{3,\skriw}=U_{1,\skriw}\cup U_{2,\skriw}.$     Now, observe  from  Theorem~\ref{main1a}  that
		
		$$\lim_{\lambda\to\infty}\frac{1}{\lambda}\log \P\Big\{(M_1^{Y^\lambda},M_2^{Y^\lambda})\in U_{3,\skriw}^{c}\Big \}\le -\inf_{(\beta,\phi) \in  F_{3,\skriw}^{c}}I(\beta,\phi).$$
		
		It  meets the requirement  for  the study  to  prove  that   $I$  is strictly  positive. For instance,there  is  a  sequence  $(\beta_n,\phi_n)\to(\beta,\phi)$   such  that  $I(\beta_{\lambda},\phi_{\lambda})\downarrow I(\beta,\phi)=0.$  This  means  $\beta=\mu\otimes \skrik$  and  $\phi=t\mu\otimes \skrik\times \mu\otimes \skrik$  which  contradicts  $(\beta,\phi)\in U_3^{c}.$  This  ends  the  proof of  the  Lemma.
	\end{proof} 
	
We write $M_{\Delta}^{\lambda}=\frac{1}{\lambda}\sum_{u\in I} \delta_{(\sigma^\lambda,\sigma^\lambda)}$  and  observe  that the  distribution  of  the  marked  SINR random network  $P(y)=\P\Big\{Y^{\lambda}=y\Big\}$  is  given  by  $$P_{\lambda}(u)=\prod_{u=1}^{I}|\mu\otimes \skrik(y_u,\ell_u)\prod_{(u,v)\in E}\frac{T^{r^\lambda}([y_u,\ell_u],[y_v,\ell_v])}{1-T^{\lambda}([y_u,\ell_u],[y_v,\ell_v])}\prod_{(u,v)\in \skrie} (1-T^{\lambda}([y_u,\ell_u],[y_v,\ell_v]))\prod_{x=1}^{I}(1-T^{\lambda}([y_u,\ell_u],[y_v,\ell_v]))$$
	
	$$\begin{aligned}-\frac{1}{a_{\lambda}\lambda^2\log \lambda}\log P_{\lambda}(y)&=\frac{1}{a_{\lambda}\lambda\log \lambda}\Big\langle -\log \mu\otimes T\,,M_1^{Y^\lambda}\Big\rangle +\frac{1}{\log \lambda}\Big \langle -\log \Big(\sfrac{T^{r^\lambda}}{1-T^{r^\lambda}}\Big ) \,,M_{2}^{Y^\lambda}\Big \rangle\\
		& +\frac{1}{a_{\lambda}\log \lambda}\Big\langle-\log (1-T^{r^\lambda})\,,M_1^{Y^\lambda}\otimes M_{1}^{Y^\lambda}\Big\rangle +\frac{1}{a_{\lambda}\lambda\log\lambda}\Big\langle-\log (1-T^{\lambda})\,,M_{\Delta}^{\lambda}\Big\rangle
	\end{aligned}  $$
	
	Notice,        $$\displaystyle \lim_{\lambda\to\infty}\frac{1}{a_{\lambda}\lambda\log \lambda}\Big\langle -\log \mu\otimes \skrik\,,M_1^{Y^\lambda}\Big\rangle=\lim_{\lambda\to \infty}\frac{1}{\lambda}\Big\langle-\log (1-T^{\lambda}\,,M_{\Delta}^{\lambda}\Big\rangle=\lim_{\lambda\to\infty}\frac{1}{a_{\lambda}\log \lambda}\Big\langle-\log (1-T^{r^\lambda})\,,M_1^{Y^\lambda}\otimes M_{1}^{Y^\lambda}\Big\rangle=0.$$
	
	Using,  Lemma~\ref{WLLN}  we  have  
	
	$$\lim_{\lambda\to\infty}\frac{1}{\log \lambda}\Big \langle -\log \Big(T^{\lambda}/(1-T^{\lambda}\Big) \,,M_{2}^{Y^\lambda}\Big \rangle=\Big \langle \1 \,, t\mu\otimes \skrik\times \mu\otimes \skrik\Big \rangle$$
	
	which  concludes  the  proof  of  Theorem~\ref{main1b}.
	
	\section{Proof  of  Theorem~\ref{main1c}\label{Sec5}, Corollary~\ref{cardinality}, Corollary~\ref{main2d}}
	For  $\beta\in\skrim(\skriw)$  we  define  the  spectral  potential  of  the  marked  SINR graph $(Y^{\lambda})$  conditional  on  the  event  $\big\{M_1^{Y^\lambda}=\beta \big\},$   $\rho_t(q,\beta) $  as
	
	\begin{equation}\label{LLDP.equ1}
		\rho_{t}(q,\beta)= \Big \langle -(1-e^q)\,,\,t\beta\otimes \beta \Big \rangle.
	\end{equation}
	
	Note that  remarkable  properties  of  a  spectral  potential,  see \cite{BIV15}  or \cite{SAD2020}  holds  for  $\rho_{t} $.
	
	For $\beta\in\skrim(\skriw\times\skriw)$, we  observe that  $I_{\beta}( \phi)$ is the  Kullback  action  of  the marked  SINR graph $Y^{\lambda}$.
	
	\begin{lemma}\label {LLDP.equ2} The  following  hold  for the  Kullback  action or divergence function $I_{\beta}(\phi)$:
		\begin{itemize}
			\item $$I_{\beta}( \phi)=\sup_{g\in \skric}\big\{\langle g,\, \phi\rangle-\phi_{t}(g,\beta) \big\}$$ 
			\item  The  function  $I_{\beta}( \phi) $  is  convex  and  lower semi-continuous   on  the  space  $\skrim(\skriw\times\skriw).$
			\item For  any real  $\alpha$,  the  set  $\Big\{ \phi\in \skrim(\skriw\times\skriw): \, I_{\beta}( \phi)\le \alpha \Big\}$  is  weakly  compact.
		\end{itemize}
	\end{lemma}
	The  proof  of  Lemma~\ref{LLDP.equ2} is  excluded from  the  article. Scholars of interest  may infer  to  \cite{SKAD2020}  for  likewise proof for empirical  measures of ` the  supercritical marked SINR random network processes  and/or the  references therein  for proof of  the  lemma for  empirical measures  on  measurable  spaces.\\

	Note  from  Lemma~\ref{LLDP.equ2}  that,  for  any  $\eps>0$, there  exists some  function  $q\in\skriw\times\skriw$   such  that  
	$$I_{\beta}(\phi)-\sfrac{\eps}{2}< \langle q\,,\, \beta \rangle -\phi_{t}(q,\phi).$$
	
	We define  the  probability  distribution  of  the  powered $R$   by $P_{\beta} $  by  
	
	$$P_{\beta}(y)=\prod_{(u,v)\in E}e^{q(u,v)}\prod_{(u,v)\in \skrie}e^{g_\lambda(u,v)}	,$$
	
	where  $$g_{\lambda}(u,v)=\frac{1}{ a_{\lambda}}\log\Big[1-T^{\lambda}(u,v)+T^{\lambda}(u,v)e^{q(u,v)}\Big]$$
	Then,  clearly  that  
	
	$$\begin{aligned}
		\frac{dP_{\beta}}{d\tilde{P}_{\beta}}(y)&=\prod_{(u,v)\in E}e^{-q(u,v)}\prod_{(u,v)\in \skrie} e^{-g_{\lambda}(u,v) a_{\lambda}}\\
		&= e^{-\lambda^2 a_{\lambda} (\langle\sfrac{1}{2} q,M_2^{Y^\lambda}\rangle-\lambda^2 a_{\lambda} \langle \sfrac{1}{2}h_{\lambda},M_1^{Y^\lambda}\otimes M_1^{Y^\lambda}\rangle )+\langle \sfrac{1}{2} g_{\lambda}, M_{\Delta}^{\lambda} \rangle }
	\end{aligned}$$

	Now  define  the  neighbourhood  of  $\phi,$    $B_{\phi}$ by
	$$B_{\phi}:=\Big\{\omega\in \skrim(\skriw\times\skriw): \, \langle q, \omega \rangle-\rho_t(q,\beta) >\langle q, \phi\rangle -\rho_t(q,\phi)-\eps/2 \Big \}$$

	Note  that  under the  condition  $M_{2}^{Y^\lambda}\in B_{\phi}$   we  have

	$$\begin{aligned}
		\frac{dP_{\beta}}{d\tilde{P}_{\beta}}(y)< e^{-\lambda^2 a_{\lambda} (\langle\sfrac{1}{2} g,L_2^{\lambda}\rangle-\lambda^2 a_{\lambda} \langle \sfrac{1}{2}h_{\lambda},M_1^{Y^\lambda}\otimes M_1^{Y^\lambda}\rangle )+\langle \sfrac{1}{2} h_{\lambda}, M_{\Delta}^{\lambda} \rangle }<e^{-\lambda^2 a_{\lambda}  I_{sc}(\nu)+\lambda^2a_{\lambda} \eps} \end{aligned}$$
	
	Thus,  the study can deduce that 
	
	$$P_{\beta }\Big\{Y^{\lambda }\in\skrig\Big | M_2^{Y^\lambda}\in D_{\phi}\Big\}\le \int \1_{\{M_2^{Y^\lambda}\in D_{\phi}\}} d\tilde{P}_{\beta}(Y^{\lambda})\le \int e^{-\lambda^2 a_{\lambda} I_{sc}(\beta)-\lambda\eps}d\tilde{P}_{\beta}(Y^{\lambda}) \le e^{-\lambda^{2}a_{\lambda} I_{sc}(\phi)-\lambda^2 a_{\lambda}\eps}.$$.
	
	Given that $I^2 (\phi)=0$ means  Theorem ~\ref{main1b} (ii),  hence  it is enough us to obtain  that  the  result is  true  for  a probability  distribution  of  the form $\phi=e^{q}\beta\otimes\beta$  and  for  $I^2(\phi)=\sfrac{1}{2}\skrih(\phi\|t\beta\otimes\beta),$ where  $\beta=\mu\otimes\skrik$.  Fix  any  number  $\eps>0$  and  any  neigbourhood $B_{\phi}\subset \skrim(\skriw\times\skriw)$.  Now  define  the  sequence  of  sets  $$\skrig^{\lambda}=\Big\{ y^{\lambda}\in \skrig: M_{2}^{y^\lambda}\in B_{\phi}\Big |\langle q, M_{2}^{y^\lambda}\rangle-\rho_{t}(q,\beta)\Big |\le \sfrac{\eps}{2}\Big\} .$$     
	
	Note that  for  all  $q\in \skrig^{\lambda}$ we  have    
	
	$$\begin{aligned}
		\frac{dP_{\beta}}{d\tilde{P}_{\beta}}(y)> e^{-\lambda^2 a_{\lambda}\langle\sfrac{1}{2} q,\phi\rangle+\lambda^2 a_{\lambda}\phi_{t}(q,\,\beta)+\lambda^{2} a_{\lambda}\sfrac{\eps}{2}}\end{aligned}.$$
	
	This  yields  
	$$P_{\pi}(\skrig^{\lambda}) =\int_{\skrig^{\lambda}}dP_{\beta}(y)\ge\int e^{-\lambda^2 a_{\lambda}\langle\sfrac{1}{2} g,\nu\rangle+\lambda^2 a_{\lambda}\rho_{t}(g,\,\beta)+\lambda^2 a_{\lambda} \sfrac{\eps}{2}}d\tilde{P}_{\beta}(y)\ge e^{-\lambda^2 a_{\lambda}\sfrac{1}{2}\skrih(\nu\|t\beta\otimes\beta)+\lambda^2 a_{\lambda} \eps}\tilde{P}_{\beta}(\skrig^{\lambda}).$$
	Applying  the law  of  large  numbers, we  have  that  $\lim_{\lambda\to\infty}\tilde{P}_{\beta}(\skrig^{\lambda})=1.$  This  completes  of  the  Theorem.\\

	{\bf  Proof of  Corollary~\ref{cardinality}}
	
	The  proof  of  Corollary~\ref{cardinality}  follows  from  the  definition  of  the  Kullback action  and  Theorem~\ref{main1c} if  we  set   $\beta=\mu\otimes \skrik$  and  $\lambda\beta\otimes\beta(a,b)=\|\lambda\beta\otimes\beta\|,$  for  all  $(a,b)\in\skriy\times\skriy.$\\

	{\bf  Proof of  Corollary~\ref{main2d}}
	
	In this scenario, the result was obtained by  Lemma~\ref{Com4} the  law  of   empirical link measure is exponentially  tight. Moreover, without loss of generality, we can assume that the
	set $U$  in Corollary~\ref{main2d}(ii) above is relatively compact. If the study chooses any $\eps> 0$; then for each functional  $\phi\in U$ the researchers can find a weak neighborhood such that the estimate of Theorem~\ref{main1c}(u) above holds. From
	all these neighborhood, the study select a finite cover of $\skrig$ and sums up over the value in Corollary~\ref{main2d}(u) above to obtain   	$$\limsup_{\lambda\to\infty}\frac{1}{\lambda}\log \P_{\beta}\Big\{Y^{\lambda}\in \skrig\,\Big|\, M_2 ^{\lambda}\in U\Big\}\le -\inf_{\phi\in U}I_{\beta}(\phi)+\eps,\,\,\mbox{ where $\beta=\mu\otimes \skrik$.}$$

	As $\eps$  was arbitrarily chosen and the lower bound in Theorem~\ref{main1a}(ii) means in the lower bound in
Theorem~\ref{main2d} holds, the study obtains  the desired results which completes the proof.
	
	\section{Conclusion}\label{Sec6}

	The study provided a   joint  large  deviation  principle  for  the  empirical  power  measure  and  the  empirical connectivity  measure   of  telecommunication  networks in   the  $\tau-$  topology. Adopting the concept of the  large  deviations,we have proved Shannon-McMillian Breiman Theorem  for  the  telecommunication  network  modelled  as  the sub-critical  SINR  network  model. In addition, we have  proved  a  local  large  deviation  principle  for  the empirical connectivity  measure  given  the  empirical power  measure   and  from  this  result;we have obtained the  classical  McMillian  theorem and  for  a given PPP.  Finally, we have  obtained  an  asymptotic   bound   on  the  set of all possible  sub-critical SINR  network  processes . Conclusively,  we  have presented   large  deviation  principles for  the  sub-critical SINR  networks. Note,  that  our  results may  form  the  bases  for  designing  an  anomaly inference  algorithms  for  subcritical  wireless  telecommunication  network  models. 
	

\end{document}